\documentclass{amsart}
\usepackage{amssymb}

\newcommand{\beq}[1]{\begin{equation}\label{#1}}
\newcommand{\eeq}{\end{equation}}

\newtheorem{theorem}{Theorem}

\theoremstyle{definition}

\title{The descriptive look at the size of subsets of groups}

\author{ Taras Banakh (Lviv University),}
\author{  Igor Protasov, Ksenia Protasova (Kyiv University)}

\begin{document}

UDC 517.5

\begin{abstract}
We explore the Borel complexity of some basic families of subsets of a countable group (large, small, thin, sparse and other) defined by the size of their elements. Applying the obtained results to the Stone-\v Cech compactification $\beta G$ of $G$, we prove, in particular, that the closure of the minimal ideal of $\beta G$ is of type $F_{\sigma\delta}$.
\end{abstract}
\maketitle

Given a group $G$, we denote by ${\bf P}_{G}$ and ${\bf F}_{G}$ the Boolean algebra of all subsets of $G$ and its ideal of all finite subsets.  We endow ${\bf P}_{G}$ with the topology arising from identification (via characteristic functions) of ${\bf P}_{G}$ with $\{0,1\}^{G}$. For $K\in F_{G}$ the sets   $$ \{X\in {\bf P}_{G}: K\subseteq X\},  \  \  \{X\in {\bf P}_{G}: X\cap K=\emptyset\}$$  form the sub-base of this topology.

After the topologization, each family $\mathcal{F}$ of subsets of a group  $G$ can be considered as a subspace of  ${\bf P}_{G}$,  so one can ask about Borel complexity of $\mathcal{F}$, the question typical in the {\it Descriptive Set Theory} (see \cite{b1}). We ask these questions for the most intensively studied  families in  {\it Combinatorics of Groups.}  For the origins  of the families defined in section 1, see the survey \cite{b2}. The main results are in section 2 and its applications to $\beta G$, the Stone-\v Cech compactification  of a discrete group  $G$,  in section 3. We conclude the paper with some comments.

\section{Diversity of subsets of a group}

A subset $A$ of a group $G$ is called

\begin{itemize}
\item {\em large} \ if $G = FA $ for some  $F\in {\bf F}_{G}$;
\item {\em extralarge} \ if $A\cap L$ is large for each large subset $L$;
\item {\em small} \ if $L  \setminus A$ is large for each large subset $L$;
\item {\em thick} \ if,  for any  $F\in  {\bf F}_{G}$ there exist  $g\in G$  such that $Fg\subseteq A$;
\item {\em prethick} \ if $FA$ is thick for some  $F\in  {\bf F}_{G}$.
\end{itemize}

Some evident or easily verified (see \cite{b3}) relationships: $A$ is large if and only if  $G\setminus A$ is not thick, $A$ is small if and only if $A$ is not prethick if and only if  $G\setminus A$ is extralarge. The family of all small subsets of $G$ is an ideal in ${\bf P}_{G}$.

A subset $A$ of a group $G$ is called
\begin{itemize}
\item {\em  $P$-small\/}  if there exists  an injective sequence   $(g_{n})_{n\in\omega}$  in $G$  such  that the subsets $\{g_{n} A : n\in\omega \}$ are pairwise disjoint;
\item {\em weakly $P$-small\/}  if, for any $n\in\omega$,  there exists $g_{0},\ldots , g_{n}$ such that the subsets $g_{0}A,\ldots , g_{n}A$ are pairwise disjoint;
\item {\em almost $P$-small\/}  if there exists  an injective sequence  $(g_{n})_{n\in\omega} $ in $G$  such that  $g_{n}A\cap g_{m}A$   is finite for all distinct $n,m$;
\item {\em near $P$-small\/}  if,  for every $n\in\omega$, there exists $g_{0},\ldots , g_{n}$ such that $g_{i}A\cap g_{j}A$ is finite for all distinct  $i,j\in \{0,\ldots,n\}$.
\end{itemize}

Every infinite group $G$  contains a weakly  $P$-small set, which is not $P$-small, see   \cite{b4}. Each almost $P$-small subset  can be partitioned into two  $P$-small  subsets \cite{b5}. Every countable Abelian group contains a near  $P$-small subset which is neither weakly nor almost $P$-small \cite{b6}.

A subset $A$  of a group $G$  with the identity $e$  is called
\begin{itemize}
\item {\em  thin} \ if $gA\cap A$ is finite for  each $g\in G\setminus \{e\}$;
\item {\em sparse} if, for every infinite  subset $Y$ of $G$, there exists a non-empty finite subset $F\subset Y$ such that $\bigcap_{g\in F}gA$ is finite.
\end{itemize}

The union of two thin subsets need not to be thin, but the family of all sparse subsets is an ideal in $P _{G}$ \cite{b7}.  For plenty of modifications and generalizations of thin and sparse subsets see \cite{b5}, \cite{b8}, \cite{b9}, \cite{b10}, \cite{b11}.

\section{ Results}

For a group $G$, we denote by ${\bf L}_{G}$, ${\bf EL}_{G}$, ${\bf S}_{G}$,  ${\bf T}_{G}$, ${\bf PT}_{G}$  the sets of all large, extralarge, small, thick and prethick subsets of $G$, respectively.


\begin{theorem} For a countable group $G$, we have: ${\bf L}_{G}$ is  $F_{\sigma}$, ${\bf T}_{G}$ is $G_{\delta}$, ${\bf PT}_{G}$  is $G_{\delta\sigma}$, ${\bf S}_{G}$ and ${\bf EL}_{G}$  are $F_{\sigma\delta}$.
\end{theorem}

\begin{proof} We take $F, H \in {\bf F}_{G}$ and prove the following auxiliary claim: the set  $T(F,H)=\{A\in {\bf P}_{G}: H\subseteq FA\}$ is open.

Indeed, let $F=\{g_{1}, \ldots ,g_{n}\}$ and  $(H_{1}, \ldots ,H_{n})$ is a partition of $H$. Then the set $\{A\in {\bf P}_{G}: H_{1}\subseteq g_{1}A, \ldots ,H_{n}\subseteq g_{n}A\}$ is open. It follows that $T(F, H)$ is open.

Now, the set  $T_{H} (F)=\bigcup \{T(F, Hg): g\in G\}$ is open and the set  $T(F)=\bigcap \{T_{H}(F): H\in {\bf F}_{G}\}$  is $G_{\delta}$. We note that  ${\bf T}_{G}=T(\{e\})$ and   ${\bf PT}_{G}=\bigcup \{T(F): F\in \mathbf{F}_{G} \}$ so ${\bf T}_{G}$  is $G_{\delta}$  and ${\bf PT}_{G}$ is $G_{\delta\sigma}$.

Since  ${\bf S}_{G}={\bf P}_{G}\setminus {\bf PT}_{G}$,   ${\bf S}_{G}$  is $F_{\delta\sigma}$. The mapping defined by  $A\longmapsto G\setminus A$ is a homeomorphism of   ${\bf P}_{G}$,  so ${\bf L}_G$ is homeomorphic to  ${\bf P}_{G}\setminus \mathbf{T}_{G} $ and  ${\bf EL}_{G}$ is homeomorphic to $ {\bf S}_{G}$. Hence, $ {\bf L}_{G}$ is  $F_{\sigma}$ and  ${\bf EL}_{G}$ is $F_{\delta\sigma}$.
\end{proof}

\begin{theorem} For a countable group $G$,  the sets of thin, weakly $P$-small  and near $P$-small subsets of $G$ are $F_{\delta\sigma}$.
\end{theorem}

\begin{proof}
Given  $F\in  {\bf F}_{G}$  and $g\in G\setminus\{e\}$,   the set $X(F,g)=\{A\in  {\bf P}_{G}: ga \notin A $ for each $a\in A\setminus F\}$ is closed.
The set
$X(g)=\bigcup \{ X(F,g): F\in {\bf F}_{G} \}$ is $F_{\sigma}$,
and $\bigcap \{ X (g) : g\in G\setminus \{e\} \}$ is the set of all thin subsets.

For $n\in \omega$, $[G]^{n}$  denotes the family of all $n$-subsets of $G$. Given $F\in [G]^{n}$,   the set $Y(F)=\{A\in  {\bf P}_{G}: gA\bigcap hA=\emptyset$ for all distinct $g,h\in F\}$ is closed, and the set of all weakly $P$-small  subsets of $G$ coincides with $$\bigcap_{n\in\omega} \bigcup\{ Y(F): F\in [G]^{n}\}.$$

Given $F\in [G]^{n}$ and  $H\in {\bf F}_{G}$,   the set  $Y(F, H)=\{A\in  {\bf P}_{G}: g(A\backslash H)\bigcap h(A\backslash H)=\emptyset$  for all distinct $g,h\in F\}$  is closed, and the set of all near $P$-small  subsets of $G$ coincides with $$\bigcap_{n\in\omega} \bigcup\{ Y(F,H): F\in [G]^{n}, \  H\in {\bf F}_{G}\}.$$
\end{proof}

We recall that a topological space $X$ is {\it Polish} if $X$ is homeomorphic to a separable complete metric space.
A subset $A$ of a   Polish space $X$ is {\it analytic} if $A$ is a continuous image of some Polish space, and $A$ is {\it coanalytic} if $X\setminus A$ is analytic.

Using the classical tree technique \cite{b1} adopted to groups in \cite{b10}, we get.\vspace{3 mm}

\begin{theorem} For a countable group $G$, the ideal of sparse subsets  is coanalytic and the set of $P$-small  subsets is analytic in   ${\bf P}_{G}$.
\end{theorem}

\section{ Applications to $\beta G$}\vspace{3 mm}

Given a discrete group $G$, we identify the Stone-\v Cech compactification  $\beta G$ with the set of all ultrafilters on $G$ and consider  $\beta G$ as a right-topological semigroup (see  \cite{b12}). Each non-empty  closed subspace $X$ of $\beta G$ is determined by some filter $\varphi_{X}$ on $G$: $$X=\bigcap\{\overline{\Phi} : \Phi\in\varphi\},  \  \   \overline{\Phi}=\{p\in\beta G: \Phi\in p \}. $$

On the other hand, each filter $\Phi$ on $G$ is a subspace of  ${\bf P}_{G}$, so we can ask about complexity of $X$ as the complexity   of $\varphi_{X}$ in ${\bf P}_{G}$.

The semigroup $\beta G$  has the minimal ideal $K_{G}$ which play one of the key parts in combinatorial  applications of   $\beta G$. By \cite{b3} Theorem 1.5, the closure $cl (K_{G})$ is determined by the filter of all extralarge  subsets of $G$.  If $G$ is countable, applying Theorem 1, we conclude that  $cl (K_{G})$  has the Borel complexity  $F_{\sigma\delta}$.

An ultrafilter $p$ on $G$ is called {\it strongly  prime} if $p\notin  cl(G^{\ast} G^{\ast})$, where   $G^{\ast}$ is a semigroup of all free ultrafilters on $G$.  We put  $X=  cl(G^{\ast} G^{\ast})$  and choose the filter  $\varphi_{X}$ which determine $X$. By \cite{b7}, $A\in \varphi_{X}$ if and only if $G\backslash A$ is sparse. If $G$  is countable,  applying Theorem 3, we conclude  that $\varphi_{X}$  is coanalitic in  ${\bf P}_{G}$.

Let  $(g_{n})_{n\in\omega}$ be an injective sequence in $G$. The set $$\{g_{i_{1}} g_{i_{2}}\ldots  g_{i_{n}}: 0 \leq i_{1}< i_{2}< \ldots < i_{n}<\omega \}$$
is called an  {\it FP-set}. By the Hindman Theorem 5.8 \cite{b12}, for every finite partition of $G$, at least one cell of the partition contains an $FP$-set. We denote
by   ${\bf FP}_{G}$ the family of all subsets of $G$ containing some $FP$-set. A subset $A$ of $G$  belongs to  ${\bf FP}_{G}$ if  and only if $A$ is an element of some idempotent of  $\beta G$. By analogy with Theorem 3, we can prove that  ${\bf FP}_{G}$ is analytic in ${\bf P}_{G}$.

\section{ Comments and open questions}

1.	Answering a question from \cite{b13}, Zakrzewski proved \cite{b14} that, for a countable amenable group $G$,  the ideal of absolute null subsets has the Borel complexity  $F_{\sigma\delta}$. Each absolute null subset is small but, for every  $\epsilon >0$, there exists a small subset $A$ of $G$ such that    $\mu(A)>1-\epsilon$ for some Banach measure $\mu$ on $G$ (see \cite{b5}).
\smallskip

2.	The classification of subsets of a group by their size can be considered in much more general context of {\it Asymptology} (see \cite{b15}).  In this context, large, thick and small subsets play the parts of dense, open and nowhere  dense subsets of a uniform topological space.
For dynamical look at the subsets of a group see \cite{b16}.
\smallskip

3. The following type of subsets of a group arised in {Asymptology} \cite{b11}. A subset $A$ of a group $G$ is called {\it scattered} if $A$ has no subsets coarsely equivalent to the Cantor macrocube. Each sparse subset is scattered,
each scattered subset is small and the set of all scattered subsets of $G$ is an ideal in ${\bf P}_{G}$.

By  Theorem 1 \cite{b11}, a subset $A$ of a group $G$ is scattered if and only if $A$ contains no piecewise shifted $FP$-sets.

Let $(g_{n})_{n\in\omega}$ be an injective sequence in $G$  and let  $(b_{n})_{n\in\omega}$ is a sequence in $G$.

The set   $$\{g_{i_{1}} g_{i_{2}}\ldots  g_{i_{n}}  b_{i_{n}} : 0 \leq i_{1}< i_{2}< \ldots < i_{n}<\omega \} $$ is called a {\it piecewise  shifted $FP$-set}.

Using this combinatorial characterization and the tree technique from \cite{b10}, we can prove that the  ideal of scattered subsets of a countable group $G$ is coanalytic in ${\bf P}_{G}$.
\smallskip

4.	By \cite{b17}, every meager topological group $G$ can be represented as the product $G=CN$ of some countable subset $C$ and nowhere dense subset $N$.
Every infinite group $G$ can be represented as a union of some countable family of small subsets \cite{b3}.
\smallskip

{\it Can every infinite group $G$ be represented as the product $G=CS$ of some countable subset $C$ and small subset $S$?}

The answer is positive if either $G$  is amenable or $G$  has a subgroup of countable index.

\vspace{5 mm}
CONTACT INFORMATION

\vspace{5 mm}

T.~Banakh:\\
Faculty of Mechanics and Mathematics\\
Ivan Franko National University of Lviv\\
Universytetska 1, 79000, Lviv, Ukraine\\
t.o.banakh@gmail.com
\medskip

I.~Protasov: \\
Department of Cybernetics  \\
         National Taras Shevchenko  University of Kiev \\
         Academic Glushkov St.4d  \\
         03680 Kiev, Ukraine \\ i.v.protasov@gmail.com

\medskip

K.~Protasova:\\
Department of Cybernetics  \\
         National Taras Shevchenko  University of Kiev \\
         Academic Glushkov St.4d  \\
         03680 Kiev, Ukraine \\ ksuha@freenet.com.ua

\end{document}